\documentclass[12pt,reqno]{amsart}
\usepackage{amsmath, amssymb, latexsym, amscd, amsthm,amsfonts,amstext}
\usepackage[mathscr]{eucal}
\usepackage{xspace}
\usepackage{array, tabularx, longtable}
\usepackage{multicol,color}
\usepackage{indentfirst}
\usepackage{graphics}
\theoremstyle{plain}
\newtheorem{theorem}{Theorem}

\newtheorem{lemma}[theorem]{Lemma}

\theoremstyle{remark}
\newtheorem{remark}{Remark}
\theoremstyle{definition}
\newtheorem{define}{Definition}

\begin{document}
\date{} 

\title[Existence of parabolic boundary points]{Existence of parabolic boundary points of certain domains in $ \mathbb C^2$}

\author{Fran\c{c}ois Berteloot and Ninh Van Thu} 
\maketitle      
\begin{abstract} In this paper, the existence of  parabolic boundary points of certain convex domains in $\mathbb C^2$ is given. On the other hand, the nonexistence of parabolic boundary points of infinite type of certain domains in $\mathbb C^2$ is also shown. 
\end{abstract}

\section{Introduction}
Let $\Omega$ be a domain in $\mathbb C^n$. Denote by $Aut(\Omega)$ the group of holomorphic automorphisms of $\Omega$. The group
$Aut(\Omega)$ is a topological group with the natural topology of uniform convergence on compact sets of $\Omega$ (i.e., the compact-open topology). 

It is a standard and classical result of H. Cartan that if $\Omega$ is a bounded domain in $\mathbb C^n$ and the automorphism group of $\Omega$ is noncompact then there exist a point $x \in \Omega$, a point $p \in \partial \Omega$, and automorphisms $\varphi_j \in Aut(\Omega)$ such that $\varphi_j(x) \to p$. In this circumstance we call $p$ a {\it boundary orbit accumulation point}. The classification of domains with noncompact automorphism group relies deeply on the study the geometry of the boundary at an orbit accumulation point $p$. For instance, Wong and Rosay \cite{W},\cite{R} showed that if $p$ is a strongly pseudoconvex point, then the domain  
is biholomorphic to the ball. In \cite{B-P1}, \cite{B-P2}, \cite{B-P3}and \cite{Ber}, E. Bedford, S. Pinchuk and F. Berteloot showed that if $p$ is of finite type, then the domain is biholomorphic to the domain of the following form 
$$M_P =\{ (w,z)\in \mathbb C^2: \text{Re} w+ P(z,\bar z)<0\}$$
where $P$ is an homogeneous polynomial in $z$ and $\bar z$. Each domain $M_P$ is called a model of $\Omega$ at $p$. To prove this, they first applied the Scaling method to point out that $Aut(\Omega)$ contains a parabolic subgroup, i.e., there is a point $p_\infty\in \partial\Omega$ and a one-parameter subgroup $\{h^t\}_{t\in\mathbb R}\subset Aut(\Omega)$ such that for all $z\in \Omega$
\begin{equation}\label{eqt1}
\lim_{t\to\pm \infty} h^t(z)=p_\infty.
\end{equation}
Each boundary point satisfying (\ref{eqt1}) is called a parabolic boundary point of $\Omega$. After that, the local analysis of a holomorphic vector field $H$ which generates the above subgroup $h^t$ was carried out to show that $\Omega$ is biholomorphic to the desired homogeneous model.

We now consider a bounded domain $\Omega\subset\mathbb C^2$. Suppose that $\Omega$ is biholomorphic to the domain $D$ defined by $D=\{(w,z)\in \mathbb C^2: \text{Re} w+\sigma(z)<0\}$ with some smooth real-valued function on the complex plane. The one-parameter of translations $\{L^t\}_{t\in\mathbb R}$ given by $L^t(w,z)=(w+it,z)$ acts on the domain $D$. The transformation $\psi: D\to \Omega$ allows us to define the one-parameter group of biholomorphic mappings $\{h^t:=\psi^{-1}\circ L^t\circ \psi \}_{ t\in \mathbb R}$ acting on $\Omega$. The fisrt aim of this paper is to show that this one-parameter group is parabolic. Namely, we prove the following theorem.

\begin{theorem}\label{T1}Let $\Omega$ be a $\mathcal C^1$-smooth, bounded, strictly geometrically convex domain in $\mathbb C^2$. Let $\psi: \Omega \to D $ be a biholomorphism, where $D:=\{(w,z)\in \mathbb C^2:\text{Re} w+\sigma(z)<0\}$, $\sigma$ is a $\mathcal C^1$-smooth nonnegative function on the complex plane such that $\sigma(0)=0$. Then, there exists some point $a_\infty\in \partial\Omega$ such that $\lim\limits_{t\to +\infty}\psi^{-1}(w\pm it,z)=a_\infty  $  for any $(w,z)\in D$.
\end{theorem}
On the other hand, R. Greene and S. G. Krantz \cite{GK} suggested the following conjecture.

\noindent
{\bf Greene-Krantz Conjecture.} {\it If the automorphism group $Aut(\Omega)$ of a smoothly bounded pseudoconvex domain $\Omega \Subset \mathbb C^n$ is noncompact, then any orbit accumulation point is of finite type. } 

The main results around this conjecture are due to R. Greene and S. G. Krantz \cite{GK}, K. T. Kim \cite{Ki1}, K. T. Kim and S. G. Krantz \cite{Ki2},\cite{Ki3}, H. B. Kang \cite{Ka}, M. Landucci \cite{La}, J. Byun and H. Gaussier \cite{By}. 

Let $P_\infty(\partial \Omega)$ be the set of all points in $\partial \Omega$ of infinite type. In \cite{La}, M. Landucci proved that the automorphism group of a domain is compact if $P_\infty(\partial \Omega)$ is a closed interval on the real "normal" line in a complex space with dimension $2$. In \cite{By}, J. Byun and H. Gaussier also proved that there is no parabolic boundary point if $P_\infty(\partial \Omega)$ is a closed interval transerval to the complex tangent space at one boundary point. For the case which $P_\infty(\partial \Omega)$ is a closed curve on the boundary, is not there exist any parabolic boundary point? In \cite{Ka}, H. B. Kang showed that the automorphism group of the bounded domain $ \Omega=\{(z,w)\in \mathbb C^2: |z|^2+P(w)<1\}$ is compact, where the function $P(w)$ is smooth and vanishes to infinite order at $w=0$. In \cite{Ki3}, K. T. Kim and S. G. Kantz considered the pseudoconvex domain $\Omega\subset \mathbb C^2$ where the local defining function of $\Omega $ in a neighborhood of the point of infinite type $(0,0)$ takes the form $\rho(z)=\text{Re} z_1 +\psi(z_2,\text{Im} z_1)$. They pointed out that the origin is not a parabolic boundary point ( see \cite[Theorem 4.1]{Ki3}). Their proof based on the "fact" that the function $\psi$ vanishes to infinite order at $(0,0)$. But, in general, it is not true, e.g., $\psi(z_2, \text{Im} z_1)=e^{{-1}/{|z_2|^2}}+ |z_2|^4. |\text{Im} z_1|^2$.

The second aim of this paper is to prove the following theorem which shows that there is no parabolic boundary point of infinite type if $P_\infty(\partial \Omega)$ is a closed curve.
\begin{theorem}\label{T2}
Let $\Omega\subset \mathbb C^2$ be a bounded, pseudoconvex domain in $\mathbb C^2$ and $0\in \partial \Omega$. Assume that 
\begin{enumerate}
\item[(1)] $\partial \Omega$ is $\mathcal C^\infty$- smooth satisfying Bell's condition R, 
\item[(2)] There exists a neighborhood $U$ of $0\in \partial \Omega$ such that
$$\Omega\cap U=\{(z_1,z_2)\in \mathbb C^2:\rho=\text{Re} z_1+ P(z_2)+Q(z_2, \text{Im} z_1)<0\},$$
where $P$ and $Q$ satisfy the following comditions
\begin{enumerate}
\item[(i)] $P$ is smooth, subharmonic and strictly positive at all points different
from the origin, where it vanishes at any order, i.e., $\lim\limits_{z_2 \to 0} \dfrac{P(z_2)}{|z_2|^N}=0,\; \forall N\geq 0,$
\item[(ii)] $Q(z_2, \text{Im} z_1)$ is smooth and can be written as $Q(z_2, \text{Im} z_1)=|z_2|^4|\text{Im}|^2R(z_2,\text{Im}z_1)$ with some smooth function $R(z_2,\text{Im}z_1)$.
\end{enumerate}  
\end{enumerate}  
Then, $(0,0)$ is not a parabolic boundary point.
\end{theorem}
\begin{remark}
By a simple computation, we see that $(0,0)$ is of infinite type, $(it,0)$ with $t$ small enough, are of type greater than or equal to $4$ and the other boundary points in a neighborhood of the origin, are strictly pseudoconvex.   
\end{remark}

{\bf Acknowledgement.} We would like to thank Professor Do Duc Thai and Dr Dang Anh Tuan for their precious discusions on this material.

\section{Existence of the parabolic boundary point}
\setcounter{theorem}{0}
In this section, we prove Theorem \ref{T1}. To do this, first of all, we recall some notations and some definitions.

 For two domains $D, \Omega$ in $\mathbb C^n$, denote by $Hol(D,\Omega)$ the set of all holomorphic maps from $D$ into $\Omega$. Moreover, we donote by $d(z,\Omega)$ the distance from the point $z\in \Omega$ to $\partial \Omega$ and by $\Delta$ the open unit disk in the complex plane.    
\begin{define} Let $p, q$ be two points in a domain $\Omega$ in $\mathbb C^n$ and $X$ a vector in $\mathbb C^n$.
\begin{enumerate}
\item[(a)]  The {\bf Kobayashi infinitesimal pseudometric} $F_\Omega (p, X)$ is defined by 
$$F_\Omega(p,X)=\inf \{\alpha>0|\exists g\in Hol(\Delta,\Omega), g(0)=0, g'(0)=X/\alpha\}$$
 \item[(b)] The {\bf Kobayashi pseudodistance} $k_\Omega (p, q)$ is defined by 
 $$k_\Omega (p, q) = \inf \int_a^b F_\Omega(\gamma (t), \gamma'(t))dt, $$
where the infimum is taken over all differentiable curves $\gamma : [a,b]\to \Omega$ such that $\gamma(a)=p$ and $\gamma(b)=q$.
\end{enumerate}
\end{define}
 Before proceeding to prove the Theorem \ref{T1}, we need some following lemmas.
\begin{lemma}\label{Lem3} Let $\Omega$ be a $\mathcal C^1$-smooth, bounded, strictly geometrically convex domain in $\mathbb C^2$. Then, there exists $\epsilon_0 >0$ such that for any $\eta\in \partial \Omega$ and for any $\epsilon \in (0,\epsilon_0]$, there exists a constant $K(\epsilon)>0$ such that the following holds.
 $$k_\Omega(z,w)\geq -\frac{1}{2}\ln d(z,\partial\Omega)-K(\epsilon)$$
for any $z,w\in\Omega$ with $|z-\eta|<\epsilon, \ |w-\eta|>3\epsilon$.
\end{lemma}
\begin{proof}
Since $\partial\Omega$ is strictly geometrically convex, there exists a family of holomorphic peak functions
\begin{equation*}
\begin{split}
F: \Omega\times \partial\Omega & \to    \mathbb C  \\
             (z,\eta)  & \mapsto   F(z,\eta)
\end{split}
\end{equation*}
such that
 \begin{enumerate}
\item[(i)] $F$ is continuous and $F(.,\eta)$ is holomorphic;
\item[(ii)] $|F|<1$;
\item[(iii)] There exist a positive constant $A$ and a positive constant $\epsilon_0$ such that $|1-F(\eta+t\Vec{n}_\eta,\eta)|\leq At$ for $t\in [0,\epsilon_0]$, where $\Vec{n}_\eta$ is the normal to $\partial\Omega$ at $\eta$.
\end{enumerate}  
Taking $\epsilon_0>0$ small enough, we may assume that $\partial B(\eta,3\epsilon)\cap \partial\Omega\ne \emptyset$ for $\epsilon\leq \epsilon_0$ and for any $\eta\in \partial\Omega$.

Let $\gamma$ be some smooth part in $\Omega$ such that $\gamma(0)=z,\ \gamma(1)=w$ and $\int\limits_0^1 F_\Omega[\gamma(t),\gamma'(t)] dt\leq k_\Omega(z,w)+1$. Let $z_0\in \gamma$ such that $|z_0-\eta|=3\epsilon$. We have

\begin{equation}\label{e-q4}
 k_\Omega(z,w)\geq \int_0^1 F_\Omega[\gamma(t),\gamma'(t)] dt-1\geq  k_\Omega(z,z_0)-1.
\end{equation}
 Let $\tilde \eta\in \partial\Omega$ be such that $z=\tilde \eta+t\Vec{n}_\eta $, $t>0$. We set $u_0:=F(z_0,\tilde{\eta})$ and $u:=F(z,\tilde \eta)$, $u$ and $u_0$ are in the unit disk $\Delta$. Then we have 

\begin{equation}\label{e-q5}
 k_\Omega(z,z_0)\geq k_\Delta(u,u_0)=\dfrac{1}{2} \ln \frac{1+|\tau_{u_0}(u)|}{1-|\tau_{u_0}(u)|}\geq - \dfrac{1}{2} \ln ( 1-|\tau_{u_0}(u)|),
\end{equation}
where $\tau_{u_0}(u)=\dfrac{u-u_0}{1-\bar{u}_0u}$.
One easily checks that
\begin{equation}\label{e-q6}
 1-|\tau_{u_0}(u)|\leq \frac{2|1-u|}{1-|u_0|}.
\end{equation}
Using the properties of $F$ we have
\begin{equation}\label{e-q7}
 |1-u|=|1-F(z,\tilde\eta)|\leq At=A d(z,b \Omega).
\end{equation}
Since $|\eta-\tilde \eta|\leq |\eta-z|+|z-\tilde\eta|<2\epsilon$ and $|z_0-\eta|=3\epsilon$, we have $|z_0-\tilde \eta|\leq \epsilon$. 

Let $M(\epsilon)=\sup\limits_{\substack{\eta\in \partial\Omega, z\in\Omega\\
                                               |z-\eta|\geq \epsilon    } }|F(z,\eta)|$, $M(\epsilon)<1$ and therefore
\begin{equation}\label{e-q8}
 1-|u_0|=1-|F(z_0,\tilde\eta)|\geq 1-M(\epsilon)>0.
\end{equation}
From (\ref{e-q6}),(\ref{e-q7}) and (\ref{e-q8}) we get
\begin{equation}\label{e-q9}
 1-|\tau_{u_0}(u)|\leq \frac{2A}{1-M(\epsilon)}d(z,\partial\Omega).
\end{equation}
Then form (\ref{e-q4}),(\ref{e-q5}) and (\ref{e-q9}) we obtain
\begin{equation}\label{e-q10}
 k_\Omega(z,w)\geq -\frac{1}{2} \ln d(z,\partial\Omega)-\frac{1}{2}\ln \frac{2A}{1-M(\epsilon)}-1
\end{equation}
and this completes the proof.
\end{proof}
\begin{lemma}\label{Lem2} Let $\Omega$ be a $\mathcal C^1$-smooth, bounded, strictly geometrically convex domain in $\mathbb C^2$ and let $\eta,\eta'\in \partial\Omega$ with $\eta\ne\eta'$. Then there exist $\epsilon>0$ and a constant $K$ suth that 

$$k_\Omega(z,w)\geq -\frac{1}{2}\ln d(z,\partial\Omega)-\frac{1}{2}\ln d(w,\partial\Omega)-K,$$
for any $z\in B(\eta,\epsilon)$ and any $w\in B(\eta',\epsilon)$.
\end{lemma}
\begin{proof}

Let $\eta $ and $\eta'$ be two distinct points on $\partial\Omega$. Suppose that $|z-\eta|<\epsilon$ and $|w-\eta'|<\epsilon$ and let $\gamma$ be a $C^1$ part in $\Omega$ connecting $z$ and $w$ such that $k_\Omega(z,w)\geq \int\limits_0^1 F_\Omega[\gamma(t),\gamma'(t)] dt-1$. If $\epsilon $ is small enough we may find $z_0\in \gamma$ such that $|z_0-\eta|>3\epsilon$ and $|z_0-\eta'|>3\epsilon$. Let $z_0=\gamma(t_0)$, then 
\begin{equation*}
\begin{split}
k_\Omega(z,w)&\geq \int\limits_0^{t_0}F_\Omega(\gamma(t),\gamma'(t)) dt + \int_{t_0}^1 F_\Omega(\gamma(t),\gamma'(t)) dt-1   \\
&\geq k_\Omega(z,z_0)+k_\Omega(z_0,w)-1\\
&\geq -\frac{1}{2}\ln d(z,\partial\Omega)-\frac{1}{2} \ln d(w,\partial\Omega)-2 K(\epsilon)-1,
\end{split}
\end{equation*}
where the last inequality is obtained by applying two times Lemma \ref{Lem3} 
\end{proof}
We now recall the difinition of horosphere
\begin{define} Let $a\in \Omega,\ \eta\in \partial\Omega,\ R>0$. The big horosphere with pole $a$, center $\eta$ and radius $R$ in $\Omega$ is defined as follows   
$$F^{\Omega}_a(\eta,R)=\{z\in \Omega: \liminf_{w\to \eta}\big[k_\Omega(z,w)-k_\Omega(a,w)\big]<\frac{1}{2}\ln R\/\}.$$
\end{define}
\begin{lemma}\label{Lem1} If $\Omega$ is a $\mathcal C^1$-smooth, bounded, strictly geometrically convex domain in $\mathbb C^2$, then $\overline{F^{\Omega}_a(\eta,R)}\cap \partial\Omega\subset\{\eta\}$ for any $a\in \Omega,\ \eta\in \partial\Omega,\ R>0$.
\end{lemma}
\begin{proof} If there exists $\eta'\in \partial\Omega\cap \overline{F^{\Omega}_a(\eta,R)}$ then we can find a sequence $\{z_n\}\subset \Omega$ with $z_n \to \eta'$ and a sequence $\{w_n\}\subset \Omega$ with $z_n \to \eta$ such that 
\begin{equation}\label{e-q1}
k_\Omega(z_n,w_n)-k_\Omega(a,w_n)<\frac{1}{2}\ln R
\end{equation}
By Lemma \ref{Lem2}, the following estimate occurs if $\eta\ne\eta'$ and $n$ great enough.

\begin{equation}\label{e-q2}
k_\Omega(z_n,w_n)\geq -\frac{1}{2}\ln d(z_n,\partial\Omega)-\frac{1}{2}\ln d(w_n,\partial\Omega)-K,
\end{equation}
 where $K$ is a constant.

On the other hand, we have 
\begin{equation}\label{e-q3}
k_\Omega(a,w_n)\leq -\frac{1}{2}\ln d(w_n,\partial\Omega)+K(a),
\end{equation}
 since $\partial\Omega$ is smooth. 

From (\ref{e-q1}), (\ref{e-q2}) and (\ref{e-q3}) we get

\begin{equation}\label{e-q31}
-\frac{1}{2}\ln d(z_n,\partial\Omega)\lesssim 1,
\end{equation}
 which is absurd.
\end{proof}

\begin{proof}[Proof of Theorem \ref{T1}]
Set $a_n:=\psi^{-1}(-t_n,0)$ where $\lim t_n=+\infty$, after taking a subsequence we may assume that $\lim a_n=a_\infty\in \partial\Omega$. We may also assume that $a_\infty$ is the origin in $\mathbb C^2$. 

Set $b_t:=\psi^{-1}(-1+it,0)$, according to Lemma \ref{Lem1}, it suffices to show that there exists $R_0>0$ such that 
\begin{equation}\label{e-q11}
\{b_t: t\in \mathbb R \}\subset F^\Omega_{a_0}(a_\infty, R_0).
\end{equation}
Since $a_n\to a_\infty$, we have 
\begin{equation}\label{e-q12}
\liminf_{w\to a_\infty}\big[k_\Omega(b_t,w)-k_\Omega(a_0,w)\big]\leq \liminf_{n\to +\infty}\big[k_\Omega(b_t,a_n)-k_\Omega(a_0,a_n)\big].
\end{equation}
Then by the invariance of the Kobayashi metric and the convexity of $D$ we have  
\begin{equation}\label{e-q13}
\begin{split}
k_\Omega(b_t,a_n)-k_\Omega(a_0,a_n)&=k_D[(-1+it,0),(-t_n,0)]-k_D[(-t_0,0),(-t_n,0)]\\
                                      &=k_H(-1+it,-t_n)-k_H(-t_0,-t_n),
\end{split}
\end{equation}

where $\mathrm H $ is the left half plane $\{\text{Re} w<0\}$.

Let $\sigma:\mathrm H\to \Delta$ be a biholomorphism between $\mathrm H$ and the disk $\Delta$ given by $\sigma(w)=\dfrac{w+1}{w-1}$. Set $z_t:=\sigma(-1+it) =\dfrac{it}{-2+it}$ and $x_n:= \sigma(-t_n)=\dfrac{-t_n+1}{-t_n-1}$. Then we have
\begin{equation}\label{e-q14}
\begin{split}
&k_{\mathrm H}(-1+it,-t_n)-k_{\mathrm H}(-t_0,-t_n)=k_\Delta (z_t,x_n)-k_\Delta (x_0,x_n)\\
&=\ln\Big(\frac{|1-x_nz_t|+|x_n-z_t|}{|1-x_nz_t|-|x_n-z_t|}.\frac{|1-x_nx_0|+|x_n-x_0|}{|1-x_nx_0|-|x_n-x_0|}\Big)\\
&=\ln\Big(\frac{|1-x_nx_0|+|x_n-x_0|}{|1-x_nz_t|-|x_n-z_t|}.\frac{|1-x_nz_t|+|x_n-z_t|}{|1-x_nx_0|-|x_n-x_0|}\Big)\\
&=\ln\Big(\frac{|1-x_nx_0|^2-|x_n-x_0|^2}{|1-x_nz_t|^2-|x_n-z_t|^2}.\Big[\frac{|1-x_nz_t|+|x_n-z_t|}{|1-x_nx_0|-|x_n-x_0|}\Big]^2\Big)\\
&=\ln\Big(\frac{1-x_0^2}{1-|z_t|^2}.\Big[\frac{|1-x_nz_t|+|x_n-z_t|}{|1-x_nx_0|-|x_n-x_0|}\Big]^2\Big).
\end{split}
\end{equation}
From (\ref{e-q13}) and (\ref{e-q14})we have
\begin{equation}\label{e-q15}
\begin{split}
\lim_{n\to\infty}\big[k_\Omega(b_t,a_n)-k_\Omega(a_0,a_n)\big]&=\ln\Big(\frac{1-x_0^2}{1-|z_t|^2}.\frac{|1-z_t|^2}{|1-x_0|^2}\Big)\\
&=\ln \frac{1-x_0^2}{|1-x_0|^2}.
\end{split}
\end{equation}
Finally, (\ref{e-q11}) follows directly from (\ref{e-q12}) and (\ref{e-q15}) when $\ln \frac{1-x_0^2}{|1-x_0|^2}<\frac{1}{2} \ln R_o$.
\end{proof}
\section{non-existence of the parabolic boundary point of infinite type} 
\setcounter{theorem}{0}
Let $\Omega$ be a domain satisfying conditions given in Theorem \ref{T2}. In this section, the non-existence of the parabolic boundary point of infinite type of $\Omega$ is proved. First of all, we need some following lemmas.
\begin{lemma}\label{l1} There do not exist $a,b\in \mathbb C$ with $\text{Re} a\ne 0$ and $b\ne 0 $ such that
\begin{equation}\label{eq1}
\text{Re} [a P(z)+b z^k P'(z)]=\gamma(z) P(z),
\end{equation}
for some $k\in \mathbb N,\; k>1$ and for every $|z|<\epsilon_0$ with $\epsilon_0>0$ small enough, where $\gamma(z)$ is smooth and $\gamma(z) \to 0$ as $z\to 0$.
\end{lemma}
\begin{proof}
Suppose that there exist $a,b\in \mathbb C$ with $\text{Re} a\ne0$ and $ b\ne 0$ such that 
\begin{equation}\label{eq2}
\text{Re} [a P(z)+b z^k P'(z)]=\gamma(z) P(z),
\end{equation}
for some $k\in \mathbb N,\; k>1$ and for every $|z|<\epsilon_0$ with $\epsilon_0>0$ small enough. This equation is equivalent to
\begin{equation}\label{eq3}
1 +\text{Re}[ \frac{b}{\text{Re} a} z^k \dfrac{P'(z)}{P(z)}]=\gamma_1(z), \quad \forall \; 0<|z|<\epsilon_0,
\end{equation}
where $ \gamma_1(z)=\gamma(z)/ \text{Re} a$. Let $F(z)=\ln P(z)$ and write $ z=r e^{i\varphi}, \; \dfrac{b}{2 \text{Re} a}=\dfrac{1}{R}e^{i\psi} $. Then, by (\ref{eq3}), we get
$$\frac{\partial F}{\partial x}(z)\cos (k\varphi+\psi)+ \frac{\partial F}{\partial y}(z) \sin(k\varphi+\psi)=-\frac{R}{r^k}+ \frac{R}{r^k}\gamma_1(z). $$
If we set $\varphi_0=\dfrac{2\pi-\psi}{k-1}$, then
$$\frac{\partial F}{\partial x}(r e^{i\varphi_0})\cos (\varphi_0)+ \frac{\partial F}{\partial y}(r e^{i\varphi_0}) \sin(\varphi_0)=-\frac{R}{r^k}+ \frac{R}{r^k}\gamma_1(re^{i\varphi_0}). $$
Let $g(r):= F(re^{i\varphi_0})$. It is easy to see that
 $$g'(r)=-\frac{R}{r^k}+ \frac{R}{r^k}\gamma_1(re^{i\varphi_0} ).$$
Let $h(r):=g(r)+\dfrac{R}{1-k} \dfrac{1}{r^k-1}$. Then 
$$h'(r)=\frac{R}{r^k}\gamma_1(re^{i\varphi_0} ).$$
We may assume that there exists $r_0$ small enough such that $|h'(r)|\leq\dfrac{R}{2r^k}$, for every $0<r\leq r_0$. Thus, we have the following estimate
\begin{equation*}
\begin{split}
|h(r)|&\leq |h(r_0)|+\big | \int_{r_0}^r|h'(r)|  dr \big|\\
&\leq |h(r_0)|+\frac{R}{2}\big | \int_{r_0}^r r^{-k}  dr \big|\\
&\leq |h(r_0)|-\frac{R}{2(k-1)} r_0^{1-k}+\frac{R}{2(k-1)} r^{1-k}.\\
\end{split}
\end{equation*}
 Hence, 
\begin{equation*}
g(r)\geq \frac{R}{k-1}r^{1-k}-|h(r_0)|+\frac{R}{2(k-1)} r_0^{1-k}-\frac{R}{2(k-1)} r^{1-k}.
\end{equation*}
It implies that $\lim\limits_{r\to 0^+}g(r)=+\infty$. This means that $P(r e^{i\varphi_0})\not \to 0$ as $r\to 0^+$. It is impossible.
\end{proof}
\begin{lemma}\label{l2} There do not exist $a,b\in \mathbb C$ with $\text{Re} a\ne 0$ and $b\ne 0 $ such that
\begin{equation}\label{eq4}
\text{Re} [a P^{n+1}(z)+b z^k P'(z)]=\gamma(z) P^{n+1}(z),
\end{equation}
for some $k\in \mathbb N,\; k>1$ and for every $|z|<\epsilon_0$ with $\epsilon_0>0$ small enough, where $\gamma(z)\to 0$ as $z\to 0$.
\end{lemma}
\begin{proof}
Suppose that there exist $a,b\in \mathbb C$ with $\text{Re} a\ne 0$ and $b\ne 0$ such that 
\begin{equation}\label{eq5}
\text{Re} [a P^{n+1}(z)+b z^k P'(z)]=\gamma(z) P^{n+1}(z),
\end{equation}
for some $k\in \mathbb N,\; k>1$ and for every $|z|<\epsilon_0$ with $\epsilon_0>0$ small enough. This equation is equivalent to
\begin{equation}\label{eq6}
1 +\text{Re}[ \frac{b}{\text{Re} a} z^k \dfrac{P'(z)}{P^{n+1}(z)}]=\gamma_1(z), \quad \forall \; 0<|z|<\epsilon_0,
\end{equation}
where $\gamma_1(z)=\gamma(z)/\text{Re} a$. Let $F(z)=\dfrac{1}{ P^{n}(z)}$ and write $ z=r e^{i\varphi}, \; \dfrac{-b}{2n \text{Re} a}=\dfrac{1}{R}e^{i\psi} $. By (\ref{eq6}), we get
$$\frac{\partial F}{\partial x}(z)\cos (k\varphi+\psi)+ \frac{\partial F}{\partial y} (z)\sin(k\varphi+\psi)=-\frac{R}{r^k}+ \frac{R}{r^k}\gamma_1(z). $$
If we set $\varphi_0=\dfrac{2\pi-\psi}{k-1}$, then 
$$\frac{\partial F}{\partial x}(r e^{i\varphi_0})\cos (\varphi_0)+ \frac{\partial F}{\partial y}(r e^{i\varphi_0}) \sin(\varphi_0)=-\frac{R}{r^k}+ \frac{R}{r^k}\gamma_1(re^{i\varphi_0}). $$
Let $g(r):= F(re^{i\varphi_0})$. Then we see that
 $$g'(r)=-\frac{R}{r^k}+ \frac{R}{r^k}\gamma_1(re^{i\varphi_0}).$$
Let $h(r):=g(r)+\dfrac{R}{1-k} \dfrac{1}{r^k-1}$. Then we may assume that there is $r_0$ small enough such that
$$|h'(r)|\leq \frac{3R}{2r^k},$$
for every $0<r\leq r_0$.
Thus, we have the following estimate
\begin{equation*}
\begin{split}
|g(r)|&\leq |g(r_0)|+\big | \int_{r_0}^r|g'(r)|  dr \big|\\
&\leq |g(r_0)|+\frac{3R}{2}\big | \int_{r_0}^r r^{-k}  dr \big|\\
&\leq |g(r_0)|-\frac{3R}{2(k-1)} r_0^{1-k}+\frac{3R}{2(k-1)} r^{1-k}.
\end{split}
\end{equation*}
Therefore, we obtain
\begin{equation*}
\begin{split}
\frac{1}{ P^{n}(re^{i\varphi_0})}&\lesssim \frac{1}{r^{1-k}},\\
 P(r e^{i\varphi_0})& \gtrsim r^{\frac{k-1}{n}}.
\end{split}
\end{equation*}
This means that $ P(re^{i\varphi_0})$ does not vanish to infinite order at $r=0$. It is a contradiction.
\end{proof}
\begin{lemma}\label{l3} There do not exist $a,b\in \mathbb C$ with $\text{Re} a\ne 0$ and $b\ne 0 $ such that
\begin{equation}\label{eq7}
\text{Re} [a P^{n+1}(z)+b z P'(z)]=\gamma(z) P^{n+1}(z),
\end{equation}
for some $n\geq 0$ and for every $|z|<\epsilon_0$ with $\epsilon_0>0$ small enough, where $\gamma(z)\to 0$ as $z\to 0$.
\end{lemma}
\begin{proof} Suppose that there exist $a,b\in \mathbb C$ with $\text{Re} a\ne0$ and $b\ne 0$ such that (\ref{eq7}) holds. We first consider the case $ n=0$. Then the equation (\ref{eq7}) is equivalent to 
\begin{equation}\label{eq8}
\text{Re} [\frac{b}{\text{Re} a} z \dfrac{\partial}{\partial z}\ln P(z)]=-1+\gamma_1(z),
\end{equation}
where $\gamma_1(z):=\gamma(z)/\text{Re} a$. Let $u(z):= \ln P(z)$ and write $\dfrac{b}{2\text{Re} a}=\alpha+i\beta,\; z=x+iy$. Then, by (\ref{eq8}), we have the following first order partial differential equation
 \begin{equation}\label{eq9}
(\alpha x-\beta y)\frac{\partial}{\partial x} u(x,y)+(\beta x+\alpha y)\frac{\partial}{\partial y} u(x,y)=-1+\gamma_1(x,y).
\end{equation}
In order to solve this partial differential equation, we need to solve the following system of differential equation. 
\[
\begin{cases}
  x'(t)=\alpha x-\beta y\\
y'(t)=\beta x+\alpha y,\, t\in\mathbb R.
\end{cases}
\]
By a simple computation, we obtain 
\begin{equation}\label{eq12}
\begin{cases}
x(t)=c_1 e^{\alpha t}\cos (\beta t)+c_2 e^{\alpha t}\sin (\beta t) \\
y(t)=-c_2 e^{\alpha t}\cos (\beta t)+c_1 e^{\alpha t}\sin (\beta t , \; t\in \mathbb R,
\end{cases}
\end{equation}
where $c_1,c_2$ are two constant real numbers.
Let $g(t):=u(x(t),y(t))$. Then $g'(t)=-1+\gamma_1(x(t),y(t))$. Thus, $g(t)=-t+ \int\limits_{t_0}^{t}\gamma_1(x(s),y(s)) ds+ t_0+g(t_0)$. From (\ref{eq12}), we get 
\begin{equation}\label{eq11}
 x^2+y^2=(c^2_1+c^2_2) e^{2\alpha t},\; t\in \mathbb R.
\end{equation}
Consider three following cases 

\noindent
{\bf Case 1.} $\alpha=0$. In this case, take $c_1=r>0,\; c_2=0$, where $r$ small enough. Then, on each small circle $\{x(t)=r\cos(t),\; y(t)= r\sin(t),\; t\in [0,2\pi]\}$, $g(t)= -t+ \int\limits_{0}^{t}\gamma_1(x(s),y(s)) ds+u(r,0)$. Taking $r$ small enough, we may assume that $|\gamma_1(x(s),y(s))|\leq1/2$ for all $s\in [0,2\pi]$. It is easy to see that $|g(2\pi)-g(0)|\geq \pi$. This is absurd since $g(2\pi)=g(0)=u(r,0)$. 

\noindent
{\bf Case 2.} $\alpha>0$. By (\ref{eq11}), $(x(t),y(t))\to 0$ as $t\to -\infty$. Then, $u(x(t),y(t))\to +\infty$ as $t\to-\infty$. This is a contradiction. 

\noindent
{\bf Case 3.} $\alpha<0$. By (\ref{eq11}), we have $(x(t),y(t))\to 0$ as $t\to +\infty$ and $t= \dfrac{1}{2\alpha} \ln \dfrac{x^2+y^2}{c_1^2+c_2^2}$. Taking $t_0>0$ big enough, we may assume that $|\gamma_1(x(s),y(s))|\leq1$ for all $s\geq t_0$. Then for all $t\geq t_0$, we have
\begin{equation*}
\begin{split}
g(t) &\geq -(t-t_0)-|\int\limits_{t_0}^{t}\gamma_1(x(s),y(s)) ds|-|g(t_0)| \\
      &\geq -(t-t_0)-|\int\limits_{t_0}^{t}|\gamma_1(x(s),y(s))| ds|-|g(t_0)|\\
      &\geq -(t-t_0)- |t-t_0|-|g(t_0)| \\
      &\geq -2(t-t_0)- |g(t_0)|.
\end{split}
\end{equation*}
Hence, for all $t\geq t_0$, we obtain
\begin{equation*}
\begin{split}
P(z(t))&\gtrsim  e^{-2t}\\
           &\gtrsim |z(t)|^{-1/{\alpha}},
\end{split}
\end{equation*}
where $z(t):=x(t)+iy(t)$. It is impossible since $P$ vanishes to infinite order at $0$.

We now consider the case $n>0$. Then the equation (\ref{eq7}) is equivalent to 
\begin{equation}\label{eq13}
\text{Re} [\frac{b}{-n \text{Re} a} z \dfrac{\partial}{\partial z} \dfrac{1}{P^n(z)}]=-1+\gamma_1(z),
\end{equation}
where $\gamma_1(z):=\gamma(z)/ \text{Re} a$. Let $u(z):= \dfrac{1}{P^n(z)}$ and write $\dfrac{b}{-2n\text{Re} a}=\alpha+i\beta,\; z=x+iy$. Then, by (\ref{eq13}), we have the following first order partial differential equation
 \begin{equation}\label{eq14}
(\alpha x-\beta y)\frac{\partial}{\partial x} u(x,y)+(\beta x+\alpha y)\frac{\partial}{\partial y} u(x,y)=-1+\gamma_1(x,y).
\end{equation}
In order to solve this partial differential equation, we need to solve the following system of differential equation. 
\[
\begin{cases}
  x'(t)=\alpha x-\beta y\\
y'(t)=\beta x+\alpha y, \, t\in\mathbb R.
\end{cases}
\]
By a simple computation, we obtain 
\begin{equation}\label{eq15}
\begin{cases}
x(t)=c_1 e^{\alpha t}\cos (\beta t)+c_2 e^{\alpha t}\sin (\beta t) \\
y(t)=-c_2 e^{\alpha t}\cos (\beta t)+c_1 e^{\alpha t}\sin (\beta t) , \; t\in \mathbb R,
\end{cases}
\end{equation}
where $c_1,c_2$ are two constant real numbers. Let $g(t):=u(x(t),y(t))$. Then $g'(t)=-1+\gamma_1(x(t),y(t))$. Thus, $g(t)=-t+ \int\limits_{t_0}^{t}\gamma_1(x(s),y(s)) ds+ t_0+g(t_0)$. From (\ref{eq15}), we get 
\begin{equation}\label{eq16}
 x^2+y^2=(c^2_1+c^2_2) e^{2\alpha t},\; t\in \mathbb R.
\end{equation}
 Consider three following cases 

\noindent
{\bf Case 1.} $\alpha=0$. In this case, take $c_1=r>0,\; c_2=0$, where $r$ small enough. Then, on each small circle $\{ x(t)=r\cos(t),\; y(t)= r\sin(t),\; t\in [0,2\pi]\}$, $g(t)= -t+ \int\limits_{0}^{t}\gamma_1(x(s),y(s)) ds+u(r,0)$. Taking $r$ small enough, we may assume that $|\gamma_1(x(s),y(s))|\leq1/2$ for all $s\in [0,2\pi]$. It is easy to see that $|g(2\pi)-g(0)|\geq \pi$. This is not possible since $g(2\pi)=g(0)=u(r,0)$. 

\noindent
{\bf Case 2.} $\alpha<0$. By (\ref{eq16}), $(x(t),y(t))\to 0$ as $t\to +\infty$. Then, $u(x(t),y(t))\to -\infty$ as $t\to-\infty$. It is a contradiction.
 
\noindent
{\bf Case 3.} $\alpha>0$. By (\ref{eq16}), we have $(x(t),y(t))\to 0$ as $t\to -\infty$ and $t= \dfrac{1}{2\alpha} \ln \dfrac{x^2+y^2}{c_1^2+c_2^2}$. Taking $t_0<0$ such that $|t_0|$ big enough, we may assume that $|\gamma_1(x(s),y(s))|\leq1$ for all $s\leq t_0$. Then for all $t\leq t_0$, we have the following estimate
\begin{equation*}
\begin{split}
g(t)&\leq -(t-t_0)+|\int\limits_{t_0}^{t}\gamma_1(x(s),y(s)) ds|+|g(t_0)| \\
       &\leq -(t-t_0)+|\int\limits_{t_0}^{t}|\gamma_1(x(s),y(s))| ds|+|g(t_0)|\\
        &\leq -(t-t_0)+ |t-t_0|+|g(t_0)| \\
          &\leq -2(t-t_0)+ |g(t_0)| .
\end{split}
\end{equation*}
Hence, for all $t\leq t_0$, we obtain
\begin{equation*}
\begin{split}
P^n(z(t))&\gtrsim  \frac{1}{-2t}\\
            &\gtrsim  \frac{-1}{\ln |z(t)|},
\end{split}
\end{equation*}
where $z(t):=x(t)+iy(t)$. This implies that 
$$\lim_{t\to-\infty}\frac{P(z(t))}{|z(t)|} =+\infty. $$
This is impossible since $P$ vanishes to infinite order at $0$.
\end{proof}
Let $F=(f,g)\in Aut(\Omega)$ be such that $F(0,0)=(0,0)$. Because of Bell's condition R of $\partial \Omega$, F extends smoothly to the boumdary of $\Omega$. Let $U$ be a a neighborhood of $(0,0)$. Then, there exists a neighborhood $V$ of $(0,0)$ such that 
\begin{equation}\label{eq19}
F(\overline{\Omega}\cap V)\subset \overline{\Omega}\cap U.
\end{equation}
The following lemma is similar to Lemma 2.5 of \cite{La}.
\begin{lemma}\label{l4} Let $F=(f,g)\in Aut(\Omega)$. Let $U, V$ be two neighborhoods of $(0,0)$ such that (\ref{eq19}) holds. Then, for any $(z_1,z_2)\in V$, 
\begin{enumerate}
\item[(i)] $g(z_1,0)=0$. 
\item[(ii)] $ f(z_1,z_2)=f(z_2)$
\end{enumerate} 
\end{lemma}
\begin{proof}

\noindent
a) Let $U, V$ be two neighborhoods of $(0,0)$ such that (\ref{eq19}) holds. Let $\gamma$ be the set of all points $(it,0)\in \partial \Omega\cap U$. By Bell's condition R, the restriction to $\partial \Omega$ of the extension of $F$ to $\overline{\Omega}$ defines a C-R automorphism of $\partial \Omega$. Since the D'Angelo type is a C-R invariant, we have $F(\gamma\cap V)\subset \gamma $. Hence, $g(it,0)=0$ and $\text{Re} f(it,0)=0$. Since $h(z_1):=g(z_1,0)\in Hol(\mathrm H)\cap \mathcal C^\infty(\overline{\mathrm H})$, $g(z_1,0)\equiv 0 $. Here, we denote $\mathrm H$ by $\mathrm H=\{z_1\in \mathbb C: \text{Re} z_1<0\}$.    

\noindent
b) A classical argument based on the Hopf's lemma shows that $(\rho\circ F)(z_1,z_2)$ is also a defining function on $V$. In particular, there exists a smooth function $k(z_1,z_2)$ which is strictly positive and such that, for any $(z_1,z_2)\in V$,
\begin{equation}\label{eq20}
\begin{split}
\text{Re} z_1&+ P(z_2)+Q(z_2, \text{Im} z_1)\\
            &             =k(z_1,z_2)\big[  \text{Re} f( z_1,z_2)+ P(g(z_1,z_2))+Q(g(z_1,z_2), \text{Im}  f( z_1,z_2)\big]. 
\end{split}
\end{equation}
We claim that for any $ N\geq 1 $ and any $ (it,0)\in \gamma\cap V$  
\begin{equation}\label{eq21}
\dfrac{\partial^N}{\partial z_2^N}\Big( \text{Re} f(z_1,z_2)+ P(g(z_1,z_2))+Q(g(z_1,z_2), \text{Im} f(z_1,z_2))\Big)\Big |_{(it,0)}=0. 
\end{equation}
In fact, for any $(it,0)\in \gamma\cap V$ we have that 
$$\text{Re} f(it,0)+ P(g(it,0))+Q(g(it,0), \text{Im} f(it,0))=0.$$

From (\ref{eq20}), it follows that
$$\dfrac{\partial}{\partial z_2}\Big( \text{Re} f(z_1,z_2)+ P(g(z_1,z_2))+Q(g(z_1,z_2), \text{Im} f(z_1,z_2))\Big)\Big |_{(it,0)}=0,$$
 which implies (\ref{eq21}) for $N=1$. Taking the $N$-th derivative with respect to $z_2$ of (\ref{eq20}) and using an inductive argument, it follows that (\ref{eq21}) holds also for any $N>1$. From a), (\ref{eq21}) and the property $(2.i)$ we get that for any $N\geq 1$ and any $(it,0)\in \partial \Omega\cap V $
\begin{equation}\label{eq22}
\dfrac{\partial^N}{\partial z_2^N} f(it,0)=0. 
\end{equation}
Using the same arguments as for (a), we see that (\ref{eq22}) implies (b).
\end{proof}
\begin{proof}[Proof of theorem \ref{T2}]
Suppose that $(0,0)\in \partial\Omega$ be a parabolic boundary point associated with a one-parameter group $\{F_\theta \}_{\theta\in \mathbb R}\subset Aut(\Omega)$. Let $H$ be the vector field generating the group $\{F_\theta \}_{\theta\in \mathbb R}$, i.e., 
$$H(z)=\dfrac{d}{d\theta} F_\theta(z) \Big |_{\theta=0}.$$
Since $\Omega$ satisfies Bell's condition R, each automorphism of $\Omega$ extends to be of class $\mathcal C^\infty$ on $\overline \Omega$. Therefore, $H\in Hol(\Omega)\cap  \mathcal  C^\infty(\overline \Omega)$. Furthermore, since $F_\theta(\partial \Omega)\subset\partial \Omega$, it follows that $H(z)\in T_z(\partial \Omega)$ for all $z\in \partial\Omega$, i.e.,
\begin{equation}\label{eq221}
(\text{Re} H) \rho(\zeta)=0,\; \forall \; \zeta \in \partial \Omega.
\end{equation}
A vector field $H\in Hol(\Omega)\cap \mathcal  C^\infty(\overline \Omega)$ satisfying (\ref{eq221}) is called to be a holomorphic tangent vector field for domain $\Omega$. Since $F_\theta(0,0)=(0,0)$, it follows from Lemma \ref{l4} that $F_\theta(z_1,z_2)=(f_\theta(z_1), z_2 g_\theta(z_1,z_2))$, where $f_\theta$ and $g_\theta$ are holomorphic on $U\cap \Omega$, where $U$ is a neighborhood of $(0,0)$. Hence, the vector field $H$ has the form
$$ H(z_1,z_2)=h_1(z_1)\dfrac{\partial}{\partial z_1}+ z_2 h_2(z_1,z_2) \dfrac{\partial}{\partial z_2},$$
where $h_1$ and $h_2$ is holomophic on $\Omega$ and, is of class $\mathcal C^\infty$ up to the boundary $\partial \Omega$. Moreover, $h_1$ vanishes at the origin. By a simple computation, we get
\begin{equation*}
\begin{split}
\frac{\partial}{\partial z_1} \rho(z_1,z_2)&= \frac{1}{2} +\frac{\partial}{\partial z_1}Q(z_2, \text{Im} z_1),\\
\frac{\partial}{\partial z_2} \rho(z_1,z_2)&= P'(z_2)+\frac{\partial}{\partial z_2}Q(z_2, \text{Im} z_1).
\end{split}
\end{equation*}
Since $H(z)$ is a tangent vector field to $\partial \Omega$, we have
\begin{equation}\label{eq23}
\begin{split}
\text{Re} \Big [&( \frac{1}{2} +\frac{\partial}{\partial z_1}Q(z_2, \text{Im} z_1))h_1(z_1)+\\
&+(P'(z_2)+\frac{\partial}{\partial z_2}Q(z_2, \text{Im} z_1)) z_2h_2(z_1,z_2) \Big ]=0,
\end{split}
\end{equation}
for all $(z_1,z_2)\in\partial \Omega$. For any $(it,0)\in\partial \Omega\cap U$, we have
\begin{equation}\label{eq231}
\text{Re} h_1(it)=0.
\end{equation}
Since $h_1\in Hol(\mathrm H)\cap \mathcal C^\infty(\overline{\mathrm H})$, where $\mathrm H$ is the left half-plane, by the Schwarz reflection principle, $h_1$ can be extended to be a holomorphic on a neighborhood of $z_1=0$. From (\ref{eq23}), it follows that, for any $(-P(z_2), z_2)\in \partial \Omega\cap U$, 
\begin{equation}\label{eq24}
\begin{split}
\text{Re} \Big [\frac{1}{2} h_1(-P(z_2))+z_2P'(z_2)h_2(-P(z_2),z_2) \Big ]=0.
\end{split}
\end{equation}
 Expanding $h_1$ and $h_2$ into Taylor series about the origin, we get $h_1(z_1)=\sum\limits_{n=0}^\infty a_n z_1^n $ and $h_2(z_1,z_2)=\sum\limits_{k=0}^\infty b_k (z_1)z_2^k $, where $a_n\in\mathbb C$, $b_k\in Hol(\mathrm H)\cap \mathcal C^\infty(\overline{\mathrm H})$, for any $n, k\in\mathbb N$. Note that $a_0=0$ since $h_1(0)=0$. If there exists an integer number $n \geq1 $ such that $\text{Re} a_n\ne 0 $, then the biggest term in $\text{Re} [\dfrac{1}{2}h_1(-P(z_2))]$ has the form $\text{Re} a_n P^{n}(z_2)$. Therefore, there exists at least $k\in \mathbb N$ such that either $b_k(0)\ne 0$ or $b_k(z_1)$ vanishes to finite order at $z_1=0$. Then the biggest term in $\text{Re} \big [ z_2P'(z_2)h_2(-P(z_2),z_2)\big] $ has the form $\text{Re} \big [b z^k_2P'(z_2) P^l(z_2)\big] $, where $b\in \mathbb C^*$, $l\in\mathbb N$. By (\ref{eq24}), there exists $\epsilon_0>0$ such that 
\begin{equation}\label{eq25}
\begin{split}
\text{Re} \big [ a_n P^{n-l}(z_2)  + b z^k_2P'(z_2) \big ]=o(P^{n-l}(z_2)),
\end{split}
\end{equation}
for all $|z_2|<\epsilon_0$. It is easy to see that $n>l$. Thus, by Lemma \ref{l1}, Lemma \ref{l2}, and Lemma \ref{l3}, we get $\text{Re} a_n=b=0$. This is a contradiction. Therefore, $\text{Re} a_n=0$ for every $n\geq1$ and thus, we can write
$h_1(z_1)=i.\sum\limits_{n=1}^\infty \alpha_n z_1^n $, where $\alpha_n\in \mathbb R,\; n=1,2,\cdots$. Let $u(z_1):= \text{Re} h_1(z_1)$. Then the function $u$ is harmonic on the left haft-plane $ \mathrm H$ and is smooth up to the boundary $ \partial \mathrm H$. By (\ref{eq231}), we have, for any real number $t$ small enough, $u(it)=0$. Moreover, $u(-t)=0$ for any $t$ small enough since $h_1(z_1)=i\sum\limits_{n=1}^\infty \alpha_n z_1^n $. Hence, by the maximum principle, we conclude that $u(z_1)\equiv 0$. Consequensely, $h_1(z_1)\equiv 0$ and hence, $H$ becomes a planar vector field. This is impossible since $\partial \Omega$ is not flat near the origin. So the proof is completed. 
\end{proof}


\begin{thebibliography}{99}
\vspace{15pt}
 \bibitem{B-P1}
E. Bedford and S. Pinchuk, {\it Domains in $\mathbb C^2$ with noncompact groups of automorphisms, }Math. USSR Sbornik 63(1989), 141-151.
 \bibitem{B-P2} 
E. Bedford and S. Pinchuk, {\it Domains in $\mathbb C^{n+1}$ with noncompact automorphism group,} J. Geom. Anal. 1 (1991), 165-191.
\bibitem{B-P3} E. Bedford and S. Pinchuk, {\it Domains in $\mathbb C^2$ with noncompact automorphism groups,} Indiana Univ. Math. Journal 47(1998), 199-222.
 \bibitem{Bell} S. Bell, {\it Compactness of families of holomorphic mappings up to the boundary\,} Lecture Notes in Math. Vol. 1268, Springer-Verlag, Berlin/New Jork, 1987, 29-42. 
\bibitem{Ber} F. Berteloot, {\it Characterization of models in $\mathbb C^2$ by their automorphism groups,}  Internat. J. Math. 5(1994), 619-634.
\bibitem{By} J. Byun and H. Gaussier, {\it On the compactness of the automorphism group of a domain}, C. R. Acad. Sci. Paris, Ser. 1341 (2005), 545-548.
\bibitem{D} J. P. D'Angelo, {\it Real hypersurfaces, orders of contact, and applications,} Ann. Math. 115 (1982), 615-637.
\bibitem{GK} R. Greene and S. G. Krantz, {\it Techniques for studying automorphisms of weakly pseudoconvex domains,} Math. Notes, Vol 38, Princeton Univ. Press, Princeton, NJ, 1993, 389-410.
\bibitem{IK} A. Isaev and S. G. Krantz, {\it Domains with non-compact automorphism group : A survey,} Adv. Math. \ {146} (1999), 1-38.
\bibitem{Ka} H. B. Kang, {\it Holomorphic automorphisms of certain class of domains of infinite type,} Tohoku Math. J. \ {46} (1994), 345-422.
\bibitem{Ki1} K. T. Kim, {\it On a boundary point repelling automorphism orbits,} J. Math. Anal. Appl. \ {179} (1993), 463-482.
\bibitem{Ki2} K. T. Kim and S. G. Krantz, {\it Convex scaling and domains with non-compact automorphism group,} Illinois J. Math. \ {45} (2001), 1273-1299.
\bibitem{Ki3} K. T. Kim and S. G. Krantz, {\it Some new results on domains in complex space with non-compact automorphism group,} 
J. Math. Anal. Appl. \ {281} (2003), 417- 424.
\bibitem{La} M. Landucci, {\it The automorphism group of domains with boundary points of infinite type,} Illinois J. Math. \ {48} (2004), 33-40.
\bibitem{W} B. Wong, {\it Characterization of the ball in $\mathbb C^n$ by its automorphism group}, Invent. Math. 41 (1977), 253-257.
\bibitem{R} J. P. Rosay, {\it Sur une caracterisation de la boule parmi les domaines de $\mathbb C^n$ par son groupe} d'{\it automorphismes}, Ann. Inst. Fourier 29 (4) (1979), 91-97.
\vskip0.3cm


Fran\c{c}ois Berteloot

Universit\'{e} Paul Sabatier MIG

Institut de Math\'{e}matiques de Toulouse. 31062

Toulouse Cedex 9, France

E-mail: berteloo@picard.ups-tlse.fr






Ninh Van Thu

Department of Mathematics

Vietnam National University at Hanoi

334 Nguyen Trai St., Hanoi, Vietnam

E-mail: thunv@vnu.edu.vn
\end{thebibliography}
\end{document}